\documentclass[11pt,reqno]{amsart}  
\usepackage{amscd}
\usepackage{amssymb,amsxtra} 
\usepackage[all,cmtip]{xy}

\newtheorem*{lemma}{Lemma}
\newtheorem*{proposition}{Proposition}

\theoremstyle{definition}

\newtheorem*{example}{Examples}

\theoremstyle{remark}

\numberwithin{equation}{section}

\newtheoremstyle{note}% name
  {3pt}%      Space above
  {3pt}%      Space below
  {}%         Body font
  {}%         Indent amount (empty = no indent, \parindent = para indent)
  {\bf}% Thm head font
  {)}%        Punctuation after thm head
  {.01em}%     Space after thm head: " " = normal interword space;
  {}%         Thm head spec (can be left empty, meaning `normal')

\theoremstyle{note}

\theoremstyle{notelist}

\def\hsp{\mspace{1mu}}
\newcommand{\Z}{\Bbb Z}
\newcommand{\N}{\Bbb N}
\newcommand{\Q}{\Bbb Q}
\newcommand{\C}{\Bbb C}
\newcommand{\R}{\Bbb R} 
\newcommand{\Zp}{\Bbb Z_p}
\newcommand{\IND}[2]{{[#1{\hsp:\hsp}#2]}}
\newcommand{\Gal}{\mathcal G}
\newcommand{\id}{\text{id}}
\newcommand{\Aut}{\text{Aut}}
\newcommand{\inv}{^{-1}}
\newcommand{\si}{\sigma}
\newcommand{\KG}{K\#G}
\newcommand{\PI}{PI}

\begin{document}

\title{Locally \PI~but not \PI~division rings of arbitrary GK-dimension}

\author{J.~C.~McConnell\ \  and \ \  A.~R.~Wadsworth}
\email{arwadsworth@ucsd.edu}
%\thanks{.} 

%    General info
%\subjclass[2000]{Primary 54C40, 14E20; Secondary 46E25, 20C20}

%\date{January 1, 2001 and, in revised form, June 22, 2001.}

%\dedicatory{This paper is dedicated to our advisors.}

%\keywords{Differential geometry, algebraic geometry}

\begin{abstract}
We give examples of locally \PI~but not \PI~division rings of GK-dimension
$n$ for every positive integer $n$.
\end{abstract}

\maketitle
\vspace{-20pt}
\section*{Introduction} 

In this note we give a construction of division algebras which are locally \PI~but
not \PI~of GK-dimension $n$ for any positive integer $n$, with center a field $F$
which may be countable or uncountable, of any characteristic.  The example is
the quotient division ring $D = q(R)$ of a twisted group ring $R = K\#G$, where 
$K$ is a $\Zp$-extension field of $F$, and $G$ is a suitable free abelian subgroup 
of rank~$n$ of the Galois group $\Gal(K/F)$.  (Terminology will be defined below.)
$R$ can also be described as an $n$-fold iterated twisted Laurent polynomial 
ring over $K$, with center $Z(R) = F$.

In \cite [Prop.~4.1]{M}, McConnell gave an example of a simple Noetherian domain 
$R$ that is locally~PI but not PI and of GK-dimension $1$ over its center. That $R$
is a twisted group ring $\KG$ where $K$~is an infinite-degree cyclotomic  extension
of the rational numbers $\Q$ and $G$ is an infinite cyclic group of
automorphisms of~$K$. It was later 
observed that the quotient division ring $D = q(R)$ is also locally PI but not PI and 
still has GK-dimension $1$ over the center.  Some years later, in 1996, Zhang gave in 
\cite[Ex.~5.7]{Z} an example of a locally PI but not PI division algebra of 
GK-dimension~$2$ over its center.  Meanwhile, McConnell showed how 
to obtain GK-dimension $n$
examples as the quotient rings of tensor products of variants of his GK-$1$ ring 
$R$.  He sketched out his approach in a note written to Lance Small in 1997.
Later, in 2012, Small asked his colleague Wadsworth to look over McConnell's note 
to clarify the infinite Galois theory being used.  He did so, and in May, 2012 wrote 
a note \lq\lq McConnell's Example" on the example with complete proofs.  The present 
article is based on that 2012 note.

In 2018 the authors learned of the preprint \cite{DBH} with gives a different 
construction of locally PI but not PI division algebras of GK-dimension 
any positive integer $n$.  We felt it would be worthwhile to make 
McConell's construction available to the mathematical community, and that 
has led to the present article.

\section{$\Bbb Z_p$ field extensions}

Let $p$ be any prime number, and let $\Zp$ denote the additive group of $p$-adic integers.
Thus, $\Zp = \varprojlim_{n\in \Bbb N}\Z/p^n\Z$.  The usual topology on $\Zp$ as a complete 
metric space
coincides with its topology as a profinite group.  In this section we recall for the reader's 
convenience some well-known facts about infinite degree Galois field extensions with 
Galois group $\Zp$. For general background on the Galois theory of infinite-degree algebraic field extensions and the topological structure of the associated Galois groups, see, e.g., \cite[\S 1.6]{G}. 

Let $F\subseteq K$ be fields.  We say that $K$ is a {\it $\Zp$-extension} of $F$ if there 
is a chain of intermediate fields $F \subseteq L_1 \subseteq L_2 \subseteq \ldots \subseteq K$
such that $K = \bigcup_{m=1}^\infty L_m$ and each $L_m$ is Galois over $F$ with 
$\IND {L_m}F = p^m$ and cyclic Galois group $\Gal (L_m/F) \cong \Z/p^m\Z$.
Note the following properties of a $\Zp$-extension:
\begin{enumerate}
\item[(i)]
$K$ is algebraic over $F$ with $\IND KF = \infty$.
\smallskip
\item[(ii)]
If $L$ is a field with $F\subsetneqq L \subsetneqq K$, then 
$L = L_m$ for some $m$.  
\smallskip
\item[(iii)]
$K$ is Galois over $F$ (since it is a direct limit of finite-degree 
Galois extensions of $F$). Moreover,  $\Gal (K/F) \cong \Zp$,
a topological group isomorphism.
Indeed, since $K = \bigcup _{m=1} ^\infty L_m$, we have
$$
\Gal(K/F) \, = \, \varprojlim_{m \in \N}\Gal(L_m/F) \,\cong \,
\varprojlim_{m\in \N} \Z/p^m \Z \, = \, \Zp.
$$
\item[(iv)] 
The fixed field $K^{\Gal(K/F)} = F$, since $K$ is Galois over $F$.
\smallskip
\item[(v)]
If $\tau \in \Gal(K/F)$, then its restriction $\tau|_{L_m}$ lies in  
$\Gal(L_m/F)$ for each $m$, since $L_m$ is Galois over $F$.
\item[(vi)]
The group $\Gal(K/F)$ is abelian, 
uncountable, and torsion-free.  This is immediate from (iii) above.
(Here is a direct proof that $\Gal(K/F)$ is torsion-free: 
If $\tau \in \Gal(K/F)$ with $\tau \ne \id_K$, then there is an $m$ with $\tau|_{L_m}
\ne \id_{L_m}$. Say $\tau|_{L_m}$ has order $p^j$ in $\Gal(L_m/F)$, with 
$1\le j\le m$.  Then, for the fixed field $L_m^{\,\,\,\tau}$ we have 
$\IND{L_m}{L_m^{\,\,\,\tau}} = p^j$, 
so $L_m^{\,\,\,\tau} = L_{m-j}$ (using (ii) above).  For any $k \ge m$, we have $\tau|_{L_k}$ 
fixes $L_{m-j}$ but not $L_{m-j+1}$.  Hence, $L_k^{\,\,\,\tau} = L_{m-j}$, so $\tau|_{L_k}$
has order $\IND {L_k}{L_k^{\,\,\,\tau}} = p^{k-(m-j)}$, which tends to infinity as $k \to \infty$.  Thus, 
$\tau$ has infinite order.)
\smallskip 
\item[(vii)]
$\Gal(K/F)$ is topologically cyclic.  That is, there is $\sigma \in \Gal(K/F)$ with 
fixed field $K^\sigma = F$. For example, take any nonidentity 
$\rho \in \Gal(L_1/F)$,
and let $\sigma$ be any extension of $\rho $ to $K$ (which exists as $K$ is normal over $F$).
Then $K^\sigma$ doesn't contain $L_1$, so we must have $K^\sigma = F$ by 
(ii)~above.  While $\langle \sigma\rangle \ne \Gal(K/F)$ (as $\Gal(K/F)$ is uncountable),
its closure $\overline{\langle \sigma\rangle}$ is all of~$\Gal(K/F)$.
\end{enumerate}

\begin{example}\hfill
\begin{enumerate}
\item[(i)]
Let $F$ be a field with $F \ne F^p$ such that $F$   contains $p^m$ different
$p^m$-th roots of unity
for every positive integer $m$. Take any $a\in F\setminus F^p$, 
 let $L_m = F(\sqrt[p^m]a)$ in some algebraic closure of $F$, and let 
 $K = \bigcup_{m=1}^\infty L_m$.  
 We have $F \subseteq L_1 \subseteq L_2\subseteq \ldots \subseteq K$, and by 
 Kummer theory each~$L_m$ is Galois over $F$ with $\Gal(L_m/F) \cong \Z/p^m\Z$
 (since $aF^{*p^m}$ has order $p^m$ in $F^*\!/F^{*p^m}$).
 More specifically, let $k$ be any field of characteristic not $p$ such that $k$ contains all
 $p^m$-th roots of unity for all $m$.  ($k$ could be countable or uncountable.) Let 
 $F = k(x)$, where $x$ is transcendental over $k$,  then $x\notin F^p$, so we could take 
 $K = \bigcup_{m=1}^\infty
 F(\sqrt[p^m]x)$.
 \smallskip
\item[(ii)]
Let $F$ be any finite field.  In an algebraic closure $\overline F$ of $F$ there is a unique 
extension field~$L_m$ of $F$
with $\IND {L_m}F = p^m$, and $L_m$  is cyclic Galois over $F$.  Moreover, 
$L_1 \subseteq L_2\subseteq \ldots\,\,$.  Let $K = \bigcup _{m=1}^\infty L_m$.  
Then, $K$ is a $\Zp$-extension of $F$.
\smallskip
\item[(iii)]
Suppose $K$ is a $\Zp$-extension of $F$.  If $E$ is a purely transcendental field 
extension of $F$, then the field 
$K \otimes_F E$  is a $\Zp$-extension of $E$.
\smallskip
\item[(iv)]
Let $F = \Q$, the rational numbers, and let $p$ be any odd prime number.  For any 
$m\in \N$, let 
$\Q_{p^m}$ be the $p^m$-th cyclotomic extension of $\Q$, i.e., 
$\Q_{p^m} = \Q(\omega_{p^m})$, where $\omega_{p^m}$ is a primitive $p^m$-th
root of unity in $\C$.  Then, $\IND {\Q_{p^m}}\Q = \varphi (p^m)= p^{m-1}(p-1)$
and $\Q_{p^m}$ is Galois over~$\Q$ with $\Gal(\Q_{p^m}\!/\Q )\cong (\Z/p^m\Z)^*$, the 
group of units of the ring $\Z/p^m\Z$; this is a cyclic group, as $p$ is odd.
Since $\Gal(\Q_{p^m}\!/\Q) \cong \Z/(p-1)\Z\times \Z/p^{m-1}\Z$, the field
$\Q_{p^m}$ has unique subfields $E_m$ and $N_m$ with 
$\IND{E_m}\Q= p-1$ and $\IND{N_m}\Q = p^{m-1}$. (In fact, $E_m = \Q_p$, 
and $\Q_{p^m}\cong N_m \otimes_\Q \Q_p$, for each $m$.)
Let $L_m = N_{m+1}$ for all $m$.  Then, $L_m \subseteq L_{m+1}$, since 
$\Q_{p^{m+2}}$~has a unique subfield of degree $p^m$ over $\Q$.  
Thus, $K = \bigcup _{m=1}^\infty L_m$ is a $\Zp$-extension of~$\Q$.  (For
$p=2$, and $m\ge 2$, one can show that each $\Q_{2^m}\cap \R$ is a cyclic Galois extension 
of $\Q$ of degree $2^{m-2}$; hence 
$\big(\bigcup _{m = 1}^\infty \Q_{2^m}\big) \cap \R$
is a $\Z_2$-extension of $\Q$. )
\end{enumerate}
\end{example}

\section {Twisted group rings}

Let $K$ be a field, let $G$ be an abelian group, and let 
$\psi\colon G \to \Aut(K)$ be a group homomorphism.  Then we can form the 
twisted group ring $K\#G$,
$$
\textstyle
K\#G \, = \, \big\{ \sum \limits_{g\in G}c_g\, g\mid \text{each $c_g\in K$ and 
almost all $c_g = 0$}\big\}, 
$$
with addition given by $\sum c_g\,g + \sum d_g\,g = \sum (c_g+d_g)\,g$ and multiplication 
determined by 
$$
(c\,g)\,(d\,h) \,=\, (c\,\psi(g) (d))\,gh.
$$ 
We are interested here only in the case where
$G$ is a free abelian group of finite rank $n\ge 1$.  Then, $K\#G$ 
can be viewed as an $n$-fold iterated
twisted Laurent polynomial ring,
$$
K\#G\, = \, K[x_1, x_1\inv, \sigma_1; \,\ldots\,; x_n, x_n\inv ,\sigma_n],
$$
where $\{x_1, \ldots, x_n\}$ is a base of $G$ as a free $\Z$-module and each 
$\sigma_i$ is given by $\psi(x_i)$ on $K$, then extended to 
$K[x_1, x_1\inv, \sigma_1; \,\ldots\,; x_{i-1}, x_{i-1}\inv ,\sigma_{i-1}]$ 
by setting $\si_i(x_j) = x_j$ for all $j < i$.  Thus, $K\#G$ is a left and right Noetherian 
Ore domain (see, e.g., \cite[Th.~4.5, p.~21]{MR}).  Note also that $K\#G$ has a natural
grading indexed by $G$:
$$
\textstyle
K\#G \, = \,\bigoplus\limits _{g\in G}(K\#G)_g\quad\text{where} \quad (K\#G) _g = Kg. 
$$
Indeed, $K\#G$ is a \lq\lq graded division ring," i.e., every nonzero homogeneous 
element is a unit.  Furthermore, every unit of $\KG$ is homogeneous.  (To see this,
choose some total ordering on $G$ to make it an ordered abelian group.  Then 
observe that if $r,s$ are any inhomogeneous elements of $\KG$, then $rs$ is also
inhomogeneous, since its lowest-degree term is the product of the lowest-degree
terms of $r$ and $s$, and likewise for the highest-degree term.)  Thus,
for the group of units,
\begin{equation*}
%\tag
\textstyle
(\KG)^* \, = \, \bigcup\limits_{g\in G} K^*g.
\tag{$*$} 
\end{equation*}
If $\psi(G) = \{ \id_K\}$, then $K\#G$ is  the (untwisted) group 
ring $K[G]$, which is commutative.  Otherwise, $K\#G$ is noncommutative.

\begin{lemma}
Assume that $G$ is a free abelian group of finite rank.
\begin{enumerate}
\item[(i)]
$Z(K\#G) = E[H]$, where $E = K^{\psi(G)}$, the fixed field of $K$ under the action of
$\psi(G)$,  and $H = \ker(\psi)$.
\smallskip
\item [(ii)]
If $\psi $ is injective, then $Z(K\#G) = K^{\psi(G)}$.
\smallskip
\item[(iii)] 
If $|\psi(G)| = k<\infty$, then $K\#G$ is a free $Z(K\#G)$-module of rank $k^2$.
\item[(iv)]
If $\psi$ is injective, then $\KG$ is a simple ring.
\end{enumerate}
\end{lemma}

\begin{proof}
(i) Since $G$ is abelian, every homogeneous component of a central element of 
$K\#G$  is also central.  Hence, $Z(K\#G)$ is a graded subring of $K\#G$. If a
homogeneous element $c\,g$ of $\KG$ is central ($c\in K$, $g\in G$), then $c$ 
commutes with all $h\in G$, so $c\in K^{\psi(G)}=E$; also, $g$ commutes with all $d\in K$,
so $g\in \ker(\psi)$.  Thus, $Z(\KG)\subseteq E\#H = E[H]$, and the reverse inclusion 
is clear.

(ii) is immediate from (i).

(iii) 
Let $E = K^{\psi(G)}$. 
If $|\psi(G)| = k<\infty$, then ${|G:H| = |G:\ker(\psi)| = k}$ and ${\IND KE  
% = \IND K{K^{\psi(G)}}
= |\psi(G)| = k}$.  Let $\{c_1, \ldots, c_k\}$ be a base of $K$ as an $E$-vector space, and let
$g_1, \ldots, g_k\in G$ be a set of representatives for the cosets of $H$ in $G$.  Then,
$\{c_i\,g_j\}_{i=1\ j=1}^{ \ \ \,\, k\ \ \ \,\, k}$ is a base of $\KG$ as a free $E[H]$-module.

(iv) Let $\{x_1, \ldots, x_n\}$ be a base of the free
abelian group $G$.  Take subgroups $G_0 =\{\id_K\}$, ${G_1= \langle x_1\rangle}$, \ldots, 
$G_i = \langle x_1,\ldots,  x_i\rangle$, \ldots, $G_n = G$, and let 
$R_i = K\#G_i$.  So $R_0 = K$ and ${R_i = R_{i-1}[x_i, x_i\inv;\si_i]}$ for $i = 1,2, \ldots, n$
where the automorphism $\si_i$ of $R_{i-1}$ is given by $\psi(x_i)$ on $K$ and $\si_i(x_j) = x_j$ for~$j<i$.
So, $R_n = \KG$.   Of course,
$R_0 = K$ is a simple ring. For $i\ge 1$, if $R_{i-1}$ is simple, then its twisted Laurent 
polynomial ring $R_i$ is simple by 
\cite[Th.~1.8.5, p.~35]{MR}, since no power of $\si_i$ is an inner automorphism of 
$R_{i-1}$. (Since $R_{i-1}^* = \bigcup_{g\in G_{i-1}}K^*g$ by ($*$) above, every  inner automophism of $R_{i-1}$ acts on $K$ by an element of $\psi(G_{i-1})$, while 
$\si_i$ acts on $K$ by $\psi(x_i)$; no power of $\psi(x_i)$ lies in $\psi(G_{i-1})$, as 
$\psi$ is injective.) Hence, by induction, $\KG = R_n$ is simple.
\end{proof}

\section{The example}

Let $p$ be any prime number, and let $F\subseteq K$ be any $\Zp$-extension of fields, as
described in \S 1.   Take any $\si_1\in \Gal(K/F)$ such that $F = K^{\si_1}$
(see (vii) in \S 1).  Now fix any positive integer $n$, and take any $\si_2, \ldots, \si_n\in \Gal(K/F)$ 
such that $\si_1, \ldots, \si_n$ are $\Z$-independent in $\Gal(K/F)$.  This is possible 
since $\Gal(K/F)$ is an uncountable torsion-free abelian group. Let 
$G = \langle \si_1, \ldots, \si_n\rangle \subseteq \Gal(K/F)$, so $G \cong \Z^n$, and let 
$\psi\colon G \hookrightarrow \Gal(K/F) $ be the inclusion map. Let $R = \KG$, and let $D$ be its 
quotient division ring, $D = q(R)$.  

\begin{proposition}
$D$ has center $Z(D) = Z(R) = F$.  Moreover, $D$ is locally \PI~but not \PI, and 
$D$~has GK-dimension $n$ as an $F$-algebra.
\end{proposition}

\begin{proof}
For any positive integer $k$, let $L_k$ be the field with ${F \subseteq L_k \subseteq K}$ and 
$\IND{L_k}F = p^k$.  Since $L_k$~is Galois over~$F$, the $\si_i$ restrict to automorphisms
of $L_k$, so we can view the twisted group ring ${S_k = L_k\#G}$ as a subring of $R$.
Then $S_1\subseteq S_2 \subseteq \ldots$ and $R = \bigcup_{k=1}^\infty S_k$.  
Hence, ${D = q(R) = \bigcup_{k=1}^\infty q(S_k)}$.  

Consider $S_k$.  Since $K^{\si_1} = F$, we also have $L_k^{\si_1} = F$.  Hence,
$\psi_k\colon G \to \Gal(L_k/F)$  (given by ${\tau \mapsto \tau|_{L_k}}$) is surjective.
Let $H_k= \ker(\psi_k)$, which is a subgroup of $G$ with ${|G:H_k| = |\Gal(L_k/F)| = p^k}$.  
So, like $G$, the subgroup $H_k$ is free abelian of rank $n$.  By the Lemma,
$Z(S_k)= F[H_k]$, which is a Laurent polynomial ring in $n$ variables over $F$.
Therefore, $q(Z(S_k))$ is a rational function of transcendence degree, so GK-dimension,
$n$ over $F$.  By the Lemma, $S_k$ is a free $Z(S_k)$-module of rank~$p^{2k}$. 
Hence, the division ring $q(S_k) = S_k\otimes_{Z(S_k)}q(Z(S_k))$ has dimension 
$p^{2k}$ over its center $q(Z(S_k))$.  Thus, $q(S_k)$ has \PI-degree $p^k$
(see \cite[13.3.6, p.~455]{MR}).  Also, 
$q(S_k)$ has GK-dimension $n$ over $F$, since it is finite-dimensional over $q(Z(S_k))$,
which has GK-dimension $n$ over $F$.

Since $D = \bigcup_{k=1}^\infty q(S_k)$, a nested union, and each $q(S_k)$ has
GK-dimension $n$ over $F$, $D$ also has GK-dimension $n$ over $F$.

Every finitely-generated $F$-subalgebra or finitely-generated division subalgebra $T$ of $D$ lies in some 
$q(S_k)$.  So, $T$~is~\PI.  Hence, $D$ is locally \PI.  But since the \PI~degree of 
$q(S_k)$ tends to infinity with $k$, $D$ cannot  be \PI\   (see, e.g., 
\cite[\S 13.3, pp.~454--456]{MR}).

By the Lemma, $Z(R) = K^{\psi(G)} = F$.  The Lemma also shows that $R$ is a
simple ring.  Therefore, $Z(D) = Z(R)$, as one can see by considering the denominator 
ideal of any element of $Z(D)$ (as in \cite[Prop. 2.1.16(viii), p.~48]{MR}).
\end{proof}

\end{document}